\documentclass[12pt,reqno]{amsart}
\usepackage[a4paper, hmargin={2.7cm,2.7cm},vmargin={3.3cm,3.3cm}]{geometry}
\usepackage{latexsym,amsmath,amssymb,amscd}
\usepackage[noadjust]{cite}
\usepackage{amsfonts}
\usepackage{amsthm}
\usepackage{enumerate}
\usepackage{amsthm}
\usepackage{todonotes}
\usepackage{cleveref}
\usepackage[all]{xy}
\usepackage{dsfont}
\usepackage{etoolbox}
\numberwithin{equation}{section}

\makeatletter
\patchcmd{\ttlh@hang}{\parindent\z@}{\parindent\z@\leavevmode}{}{}
\patchcmd{\ttlh@hang}{\noindent}{}{}{}
\makeatother

\makeatletter
\providecommand*{\cupdot}{%
  \mathbin{%
    \mathpalette\@cupdot{}%
  }%
}
\newcommand*{\@cupdot}[2]{%
  \ooalign{%
    $\m@th#1\cup$\cr
    \sbox0{$#1\cup$}%
    \dimen@=\ht0 %
    \sbox0{$\m@th#1\cdot$}%
    \advance\dimen@ by -\ht0 %
    \dimen@=.5\dimen@
    \hidewidth\raise\dimen@\box0\hidewidth
  }%
}

\providecommand*{\bigcupdot}{%
  \mathop{%
    \vphantom{\bigcup}%
    \mathpalette\@bigcupdot{}%
  }%
}
\newcommand*{\@bigcupdot}[2]{%
  \ooalign{%
    $\m@th#1\bigcup$\cr
    \sbox0{$#1\bigcup$}%
    \dimen@=\ht0 %
    \advance\dimen@ by -\dp0 %
    \sbox0{\scalebox{2}{$\m@th#1\cdot$}}%
    \advance\dimen@ by -\ht0 %
    \dimen@=.5\dimen@
    \hidewidth\raise\dimen@\box0\hidewidth
  }%
}
\makeatother

\newtheorem{theorem}{Theorem}[section]
\newtheorem{lemma}[theorem]{Lemma}
\newtheorem{proposition}[theorem]{Proposition}

\theoremstyle{definition}
\newtheorem{definition}[theorem]{Definition}

\theoremstyle{remark}

\DeclareMathOperator{\PW}{PW}

\newcommand{\Hpi}{\mathcal{H}_{\pi}}

\newcommand{\R}{\mathbb{R}}

\newcommand{\N}{\mathbb{N}}

\newcommand{\ve}{\varepsilon}

\DeclareMathOperator{\ess}{ess}

\newcommand{\gf}{g_{\lambda}}
\newcommand{\df}{h_{\lambda}}

\title{Riesz sequences near the critical density}

\author{Marcin Bownik}
\address{Department of Mathematics, University of Oregon, Eugene, OR 97403--1222, USA}
\email{mbownik@uoregon.edu}

\author{Jordy Timo van Velthoven}
\address{Faculty of Mathematics,
University of Vienna,
Oskar-Morgenstern-Platz 1,
1090 Vienna, Austria}
\email{jordy-timo.van-velthoven@univie.ac.at}

\makeatletter
\@namedef{subjclassname@2020}{\textup{2020} Mathematics Subject Classification}
\makeatother

\subjclass[2020]{42A65, 42C15, 42C40, 46B15}

\keywords{Beurling density, critical density, Riesz sequence, reproducing kernels}

\begin{document}

\maketitle

\begin{abstract}
We construct Riesz sequences  whose Beurling density is arbitrary close to the critical density in the setting of reproducing kernel Hilbert spaces on metric measure spaces. In particular, our results apply to exponential systems on unbounded sets, nonlocalized Gabor systems and general coherent systems arising from nilpotent Lie groups. The methods used in this paper are based on our previous work on the redundancy of frames and a selector form of Weaver's conjecture.
\end{abstract}

\section{Introduction}
The notion of reproducing kernels in a Hilbert space underlies many classical systems of functions in complex and harmonic analysis, such as those appearing in Paley-Wiener spaces, weighted spaces of holomorphic functions and spaces of matrix coefficients on Lie groups. In particular, they provide a convenient unifying framework for studying sampling and interpolation sets, which correspond to those sets whose associated reproducing kernels form a frame or a Riesz sequence. We refer to   \cite{olevskii2016functions, seip2004interpolation, young2001introduction} (and the references therein) for classical treatments on these topics.

The present paper is devoted to studying the existence of Riesz sequences and interpolation sets with density arbitrary close to the critical density. We study these questions in certain reproducing kernels on metric measure spaces, which is a setting considered earlier in, e.g., \cite{fuehr2017density, mitkovsi2020density, bownik2025redundancy}. This setting has the advantage that it provides a unifying framework for various of the aforementioned classical examples and allows us to obtain new results in each of these settings at once.

\subsection{Prior results}
Before describing our new results, we provide the relevant background and prior results. For a simpler presentation, we restrict our attention to Euclidean space throughout the introduction.

For a reproducing kernel Hilbert space $\mathcal{H} \subseteq L^2 (\R^n)$, we denote by $\{k_x \}_{x \in \R^n}$ the associated reproducing kernels satisfying $f(x) = \langle f, k_x \rangle$ for all $x \in \R^n$.
We will assume common assumptions on the reproducing kernels as in, e.g., \cite{fuehr2017density, enstad2025dynamical, mitkovsi2020density, bownik2025redundancy, grochenig2024necessary, grochenig2019strict}. Specifically, we assume that the kernels satisfy the following conditions:
\begin{enumerate}[-]
 \item \emph{Diagonal conditon}: There exist constants $0 < C_1 \leq C_2 < \infty$ such that
\begin{align} \label{eq:dc_intro} \tag{DC}
 C_1 \leq \| k_x \|^2 \leq C_2 \quad \text{for all} \quad x \in \mathbb{R}^n.
\end{align}

\item \emph{Weak localization}: For every $\varepsilon > 0$, there exists $r = r(\varepsilon) > 0$ such that
\begin{align} \label{eq:wl_intro} \tag{WL}
 \sup_{x \in \mathbb{R}^n} \int_{\mathbb{R}^n \setminus B_r (x)} |\langle k_x, k_y \rangle |^2 \; dy < \varepsilon^2.
\end{align}

\item \emph{Homogeneous approximation property}: If $\Lambda \subseteq \mathbb{R}^n$ is a countable set such that $\{k_{\lambda} \}_{\lambda \in \Lambda}$ is a Bessel sequence in $\mathcal{H} \subseteq L^2 (\mathbb{R}^n)$ and $\varepsilon > 0$, then there exists $r = r(\varepsilon) > 0$ such that
\begin{align} \label{eq:hap_intro} \tag{HAP}
 \sup_{x \in \mathbb{R}^n} \sum_{\lambda \in \Lambda \setminus B_r (x)} |\langle k_x, k_{\lambda} \rangle |^2  < \varepsilon^2.
\end{align}
\end{enumerate}

These conditions are satisfied for a variety of examples studied extensively in the literature. In particular, we mention Paley-Wiener spaces of functions with bounded spectrum \cite{olevskii2016functions, olevskii2012revisiting}, Gabor spaces \cite{ramanathan1995incompleteness, christensen1999density}, certain spaces of matrix coefficients of Lie groups \cite{grochenig2008homogeneous, enstad2025dynamical}, weighted spaces of entire functions \cite{grochenig2019strict} and spectral subspaces of elliptic differential operators \cite{grochenig2024necessary, grochenig2017what}.

Under the diagonal condition \eqref{eq:dc_intro}, weak localization condition \eqref{eq:wl_intro} and the homogeneous approximation property \eqref{eq:hap_intro}, the following necessary density conditions for frames and Riesz sequences were shown in \cite[Corollary 4.1]{fuehr2017density}.

\begin{theorem}[\cite{fuehr2017density}] \label{thm:necessary_intro}
Let $\mathcal{H} \subseteq L^2 (\mathbb{R}^n)$ be a reproducing kernel Hilbert space with reproducing kernels $\{k_x \}_{x \in \mathbb{R}^n}$. Suppose that $\{k_x \}_{x \in \mathbb{R}^n}$ satisfies the diagonal condition \eqref{eq:dc_intro}, the weak localization \eqref{eq:wl_intro} and the homogeneous approximation property \eqref{eq:hap_intro}. Then the following assertions hold:
\begin{enumerate}
 \item[(i)] If $\{k_{\lambda} \}_{\lambda \in \Lambda}$ is a frame for $\mathcal{H}$, then
 \[
  D^-_0 (\Lambda) := \liminf_{r \to \infty} \inf_{x \in \mathbb{R}^n} \frac{\#\big(\Lambda \cap B_r (x) \big)}{\int_{B_r(x)} k(y,y) \; dy} \geq 1.
 \]

\item[(ii)] If $\{k_{\lambda} \}_{\lambda \in \Lambda}$ is a Riesz sequence in $\mathcal{H}$, then
 \[
  D^+_0 (\Lambda) := \limsup_{r \to \infty} \sup_{x \in \mathbb{R}^n} \frac{\#\big(\Lambda \cap B_r (x) \big)}{\int_{B_r(x)} k(y,y) \; dy} \leq 1.
 \]
\end{enumerate}
In particular, if $\{k_{\lambda} \}_{\lambda \in \Lambda}$ is a Riesz basis for $\mathcal{H}$, then $D_0^+(\Lambda) = D^-_0 (\Lambda) = 1$.
\end{theorem}

The Beurling densities $D_{0}^{\pm}$ appearing in Theorem \ref{thm:necessary_intro} are weighted versions of the usual Beurling densities $D^{\pm}$ on $\R^n$, in the sense that they are Beurling densities with respect to the weighted measure $k(y,y) \; dy$.
In this dimension-free form of the necessary density conditions, the critical value of the density that separates frames and Riesz sequences is precisely  $1$.

In \cite{bownik2025redundancy}, we showed the optimality of the necessary density condition for frames in Theorem~\ref{thm:necessary_intro} by showing the following theorem:

\begin{theorem}[\cite{bownik2025redundancy}] \label{thm:frame}
Under the assumptions of Theorem \ref{thm:necessary_intro}, the following assertion holds:
For every $\varepsilon > 0$, there exists $\Lambda \subseteq \mathbb{R}^n$ satisfying
 \[
  D^+_0 (\Lambda) \leq 1 + \varepsilon
 \]
 and such that $\{k_{\lambda} \}_{\lambda \in \Lambda}$ is a frame for $\mathcal{H}$.
\end{theorem}

Theorem \ref{thm:frame} (and its more general version in \cite{bownik2025redundancy}) yielded, among others, the existence of frames near the critical density in the settings of exponential systems on unbounded sets, nonlocalized Gabor systems and spectral subspaces of ellptic differential operators. In each of these settings, the existence of a frame near the critical density was an open problem.

\subsection{New results}
In the present paper, we will complement Theorem~\ref{thm:frame} with a dual result on the existence of Riesz sequences near the critical density. So far, the existence of Riesz sequences near the critical density seems only known for exponential systems on bounded/compact sets \cite{agora2015multi, marzo2006riesz, vershynin2000coordinate}, localized Gabor systems \cite{casazza2012infinite} and weighted Fock spaces of entire functions \cite{grochenig2019strict}. Even in the special cases of exponential systems and Gabor systems, the existence of Riesz sequences near the critical density are not covered in the literature in full generality.

Our first result provides Riesz sequences near the critical density in the general setting of Theorem~\ref{thm:necessary_intro}.

\begin{theorem} \label{thm:riesz}
Under the assumptions of Theorem \ref{thm:necessary_intro}, the following assertion holds:
For every $\varepsilon > 0$, there exists $\Lambda \subseteq \mathbb{R}^n$ satisfying
 \[
  D^-_0 (\Lambda) \geq 1 - \varepsilon
 \]
 and such that $\{k_{\lambda} \}_{\lambda \in \Lambda}$ is a Riesz sequence in $\mathcal{H}$.
\end{theorem}

Theorem \ref{thm:necessary_intro} recovers, in particular, the aforementioned results on exponential systems on bounded/compact sets \cite{agora2015multi, marzo2006riesz, vershynin2000coordinate}, localized Gabor systems \cite{casazza2012infinite} and weighted Fock spaces of entire functions \cite{grochenig2019strict}. In addition, it provides a new result on the existence of interpolation sets near the critical density in spectral subspaces of elliptic differential operators \cite{grochenig2024necessary}.

Aside Theorem \ref{thm:riesz}, we show the following theorem on exponential frames for possibly unbounded sets.

\begin{theorem} \label{thm:exponential_intro}
 Let $\Omega \subseteq \R^d$ be a set of finite measure. For every $\ve > 0$, there exists $\Lambda \subseteq \R^d$ satisfying
 $
  D^- (\Lambda)   \geq (1-\varepsilon) |\Omega|
 $
and such that $\{ e^{2\pi i \lambda \cdot } \mathds{1}_{\Omega} \}_{\lambda \in \Lambda}$ is a Riesz sequence in $L^2 (\Omega)$. 
\end{theorem}

We mention that Theorem \ref{thm:exponential_intro} is optimal, in the sense that Riesz sequences $\{ e^{2\pi i \lambda \cdot } \mathds{1}_{\Omega} \}_{\lambda \in \Lambda}$ with critical density $D^+ (\Lambda) = |\Omega|$ might not exist for certain spectra $\Omega \subseteq \R^d$, cf. \cite{kozma2023set, enstad2025exponential}. In addition to the mere existence of Riesz sequence near the critical density, we also show their frame bounds can be chosen to involve the measure of the spectrum, cf. Theorem~\ref{thm:exponential}.

Lastly, we state the following theorem on nonlocalized Gabor systems.

\begin{theorem} \label{thm:gabor_intro}
Let $g \in L^2 (\mathbb{R}^d)$ be nonzero. For every $\varepsilon > 0$, there exists  $\Lambda \subseteq \mathbb{R}^{d} \times \mathbb{R}^d$ satisfying
\[
 D^- (\Lambda) \geq 1 - \varepsilon
\]
and such that $\big\{ e^{2\pi i \lambda_1 \cdot} g(\cdot - \lambda_2) \big\}_{(\lambda_1, \lambda_2) \in \Lambda}$ is a Riesz sequence in $L^2 (\mathbb{R}^d)$.
\end{theorem}

Like Theorem~\ref{thm:exponential_intro}, the conclusion of Theorem~\ref{thm:gabor_intro} is optimal for general Gabor frames. Indeed, by the Balian-Low theorem \cite{grochenig2015deformation}, there do not exist localized Gabor Riesz sequences (e.g., with a  function $g \in M^1 (\R^d)$) with critical density $D^+ (\Lambda) = 1$.

We derive both Theorem~\ref{thm:exponential_intro} and Theorem~\ref{thm:gabor_intro} from a more general version of Theorem~\ref{thm:riesz}; see Theorem~\ref{thm:main} in the main text.

\subsection{Methods}
The proof methods used in this paper are based on similar techniques as used in our paper \cite{bownik2025redundancy} and combined with Naimark's dilation theorem. Specifically, it uses various techniques for studying the overcompleteness of frames, such as the frame measure  \cite{balan2006density, balan2007measure, caspers2023overcompleteness}, and combines it with an iterative selector form of Weaver's conjecture \cite{bownik2024selector}, which builds on the solution of the Kadison-Singer conjecture \cite{marcus2015interlacing}. See also \cite{nitzan2016exponential, freeman2019discretization} for nonselector forms of iterative Weaver's conjecture.

The starting point of the proof of Theorem~\ref{thm:riesz} (resp.  Theorem~\ref{thm:main}) is Theorem~\ref{thm:frame} (resp. \cite[Theorem 5.4]{bownik2025redundancy}), which provides us with a frame near the critical density. We then remove an adequate portion of this frame that leaves us with a Bessel sequence with density below the critical density for frames, but still above the asserted density claim in Theorem~\ref{thm:riesz}. In order to obtain a Riesz sequence from this Bessel sequence, we then apply a selector form of Weaver's conjecture (see Theorem~\ref{thm:binary}) to a Naimark complement of this Bessel sequence, which allows us to extract a Riesz sequence while still controlling its Beurling density. A similar general strategy for thinning out a frame to obtain a Riesz sequence with adequate properties can be found in \cite[Corollary 6.4]{bownik2024selector}.
The challenge in this strategy for our purposes is to remove a sufficiently large portion of the frame near the critical density while still controlling its density.

\subsection{Organization} Section \ref{sec:prelim} provides preliminary results on frames and Riesz sequences and recalls the relevant results on Naimark complements and Weaver's conjecture used in the main text. The setting of reproducing kernel Hilbert spaces on metric measure spaces will be outlined in Section \ref{sec:rkhs}. In addition, this section provides various new lemmata on Beurling densities. Our main result is proven in Section \ref{sec:main}. Lastly, applications of our main results to various examples are discussed in Section \ref{sec:examples}.

\section{Preliminaries on frames and Riesz sequences} \label{sec:prelim}
Throughout this section, $\mathcal{H}$ denotes a separable Hilbert space and $I$ a countable index set.

\subsection{Frames and Riesz sequences}
A system $\{g_i\}_{i \in I}$ of vectors $g_i \in \mathcal{H}$ is said to be a \emph{Bessel sequence} in $\mathcal{H}$ if there exists a constant $B>0$, called a \emph{Bessel bound}, such that
\[
\sum_{i \in I} |\langle f, g_i \rangle |^2 \leq B \| f \|^2 \quad \text{for all} \quad f \in \mathcal{H}.
\]
Equivalently, the system $\{g_i\}_{i \in I}$ has Bessel bound $B > 0$ if it satisfies
\[
\bigg\| \sum_{i \in I} c_i g_i \bigg\|^2 \leq B \| c \|^2 \quad \text{for all} \quad c \in \ell^2 (I).
\]
Both statements are equivalent to the associated frame operator $S_I := \sum_{i \in I} \langle \cdot , g_i \rangle g_i$ being bounded on $\mathcal{H}$ with operator norm $\| S_I \| \leq B$.

A Bessel sequence $\{g_i\}_{i \in I}$ with Bessel bound $B > 0$ is called a \emph{frame} for $\mathcal{H}$ if there exists $A > 0$, called a \emph{lower frame bound}, such that
\[
A \| f \|^2 \leq \sum_{i \in I} |\langle f, g_i \rangle |^2 \leq B \| f \|^2 \quad \text{for all} \quad f \in \mathcal{H}.
\]
A frame whose frame bounds can be chosen $A = B =1$ is called a \emph{Parseval frame} for $\mathcal{H}$. If $\{g_i \}_{i \in I}$ is a frame, then its frame operator $S_I$ is invertible, and $\{ S_I^{-1} g_i \}_{i \in I}$ and $\{S_I^{-1/2} g_i \}_{i \in I}$ are the \emph{canonical dual frame} and \emph{canonical Parseval frame} associated to $\{g_i\}_{i \in I}$, respectively.

Lastly, a Bessel sequence $\{g_i\}_{i \in I}$ with Bessel bound $B>0$ is called a \emph{Riesz sequence} in $\mathcal{H}$ if there exists $A>0$, called a \emph{lower Riesz bound}, such that
\[
A \| c \|^2 \leq \bigg\| \sum_{i \in I} c_i g_i \bigg\|^2 \leq B \| c \|^2 \quad \text{for all} \quad c \in \ell^2 (I).
\]
We refer to \cite{young2001introduction} for further basic background on frames and Riesz sequences.

We will always treat a system $\{g_i \}_{i \in I}$ as an indexed family and allow for repetitions.

\subsection{Naimark's complements}

We will use the following well-known form of Naimark's dilation theoren, cf. \cite[Proposition 1.1]{han2000frames}.

\begin{lemma}[\cite{han2000frames}] \label{lem:naimark}
Let $\{g_i\}_{i \in I}$ be a Parseval frame for $\mathcal{H}$. Then $\ell^2 (I)$ contains $\mathcal{H}$ isometrically as a closed subspace such that $g_i = P e_i$ for $i \in I$, where $(e_i)_{i \in I}$ denotes the standard basis of $\ell^2 (I)$ and  $P$ denotes the orthogonal projection from $\ell^2(I)$ onto $\mathcal{H}$.
\end{lemma}

We also mention the following consequence, cf. \cite[Proposition 5.4]{bownik2019improved}.

\begin{lemma}[\cite{bownik2019improved}] \label{lem:naimark}
 Let $(e_i)_{i \in I}$ be the standard basis for $\ell^2 (I)$ and let $P : \ell^2 (I) \to \ell^2(I)$ be an orthogonal projection onto a closed subspace $\mathcal{H} \subseteq \ell^2 (I)$. Then, for any subset $J \subseteq I$ and $\delta > 0$, the following are equivalent:
 \begin{enumerate}[(i)]
  \item $\{P e_i\}_{i \in J}$ is a Riesz sequence with lower bound $\delta$;
  \item $\{(\mathbf{I} - P) e_i) \}_{i \in J}$ is a Bessel sequence with Bessel bound $1-\delta$.
 \end{enumerate}
\end{lemma}

\subsection{Completion of Bessel sequences}
The following result on the completion of Bessel sequences to Parseval frames is similar to (part of the proof of) \cite[Lemma 4.3]{bownik2025redundancy}, except for the control on the norms of the added vectors. As this control is essential for our purposes, we include its proof here.

\begin{lemma} \label{lem:bessel_completion}
Let $0 < \beta < 1$. Suppose $\{g_i\}_{i \in I}$ is a Bessel sequence in $\mathcal{H}$ with Bessel bound $1$. Then there exists a Hilbert space $\widetilde{\mathcal{H}}$ containing $ \mathcal{H}$ isometrically as a closed subspace and vectors $\varphi_n \in \widetilde{\mathcal{H}}$, $n \in \mathbb{N}$, satisfying $\| \varphi_n \|^2 = \beta$ for all $n \in \mathbb{N}$, such that the system $\{g_i \}_{i \in I} \cup \{ \varphi_n \}_{n \in \mathbb{N}}$ is a Parseval frame for $\widetilde{\mathcal{H}}$.
\end{lemma}
\begin{proof}
Let $S$ be the frame operator of the Bessel sequence $\{g_i\}_{i \in I}$. Then, by assumption, we have that $0 \leq S \leq \mathbf{I}_{\mathcal{H}}$.

Define the Hilbert space $\widetilde{\mathcal{H}} := \mathcal{H} \oplus \ell^2 (\mathbb{N})$ and identify each vector $g_i \in \mathcal{H}$ with $g_i \oplus 0 \in \widetilde{\mathcal{H}}$. The frame operator of the sequence $(g_i \oplus 0)_{i \in I}$ in $\widetilde{\mathcal{H}}$ is $\widetilde{S} := S \oplus 0$, and hence
\[
 T := \mathbf{I}_{\widetilde{\mathcal{H}}} - \widetilde{S} = (\mathbf{I}_{\mathcal{H}} - S) \oplus \mathbf{I}_{\ell^2 (\mathbb{N})}
\]
is a positive operator in $\mathcal{B}(\widetilde{\mathcal{H}})$. Moreover, its essential norm is given by
\[ \| T \|_{\ess} := \inf \big\{ \| T - K \| : K \in \mathcal{K}(\widetilde{H}) \big\} = 1. \]
Since $\beta < 1$, we get that $\| \beta^{-1} T \|_{ess} > 1$. Therefore, an application of \cite[Theorem 2]{dykema2004ellipsoidal} shows that $\beta^{-1} T$ admits a projection decomposition, i.e., there exist orthogonal projections $P_j \in \mathcal{B}(\widetilde{\mathcal{H}})$ such that $\beta^{-1} T = \sum_{j = 1}^{\infty} P_j$ with convergence in the strong operator topology. By decomposing each $P_j$ into rank-one projections yields, after enumeration, a family $\{w_n \}_{n \in \mathbb{N}} \subseteq \widetilde{\mathcal{H}}$ of unit vectors such that
\[
 \beta^{-1} T = \sum_{n = 1}^{\infty} w_n \otimes w_n
\]
with strong operator convergence.

Lastly, define $\varphi_n := \sqrt{\beta} w_n$. Then $\| \varphi_n \|^2 = \beta$ for $n \in \mathbb{N}$, and $T = \sum_{n = 1}^{\infty} \varphi_n \otimes \varphi_n$. Therefore,
\[
 \sum_{i \in I} g_i \otimes g_i + \sum_{n = 1}^{\infty} \varphi_n \otimes \varphi_n = \widetilde{S} + T = \mathbf{I}_{\widetilde{\mathcal{H}}},
\]
showing that $\{ g_i \}_{i \in I} \cup \{ \varphi_n \}_{n \in \mathbb{N}}$ is a Parseval frame for $\widetilde{\mathcal{H}}$.
\end{proof}

\subsection{Binary selectors}
For stating the selector form of Weaver's conjecture, we will need the concept of a binary selector, which was introduced in \cite[Definition 5.1 and Remark 5.5]{bownik2024selector}.

\begin{definition}\label{binse}
 Let $\{ J_{k}\} _{k \in K}$ be any partition of $I$ with $\# J_k \le 2$ for all $k$. Binary selectors of order $1$ are two disjoint sets $I_0$ and $I_1$ such that $I=I_0 \cup I_1$ and
\[
\#(I_0 \cap J_k), \; \#(I_1 \cap J_k) \le 1
\qquad\text{for all } k\in K.
\]
For $N\ge 2$, we define selectors of order $N$ inductively. Suppose that binary selectors $I_b$,  $b\in \{0,1\}^{N-1}$, of order $N-1$ are already defined. For given $b\in \{0,1\}^{N-1}$, let $\{ J_{k}\} _{k \in K}$ be any partition of $I_b$ with $\# J_k \le 2$ for all $k$. Binary selectors of order $N$ are disjoint sets $I_{b0}$ and $I_{b1}$ satisfying $I_b=I_{b0} \cup I_{b1}$ and
\[
\#(I_{b0} \cap J_k), \;  \#(I_{b1} \cap J_k) \le 1
\qquad\text{for all } k\in K.
\]
\end{definition}

The following result is \cite[Theorem 5.3]{bownik2024selector}; see also \cite{nitzan2016exponential, freeman2019discretization} for other iterative versions of Weaver's conjecture.

\begin{theorem}[\cite{bownik2024selector}] \label{thm:binary}
Let $\delta >0$ and suppose that $\{ g_i \}_{i \in I}$ is a Bessel sequence in $\mathcal{H}$ with Bessel bound $1$ satisfying
\[
 \| g_i \|^2 \leq \delta \quad \text{for all} \quad i \in I. \]
Let $N \in \mathbb{N}$ be such that $2^N < 1/\delta$. For any intermediate choices of partitions with sets of size $\le 2$, there exist binary selectors $I_b$, $b \in \{0,1\}^N$, that form a partition of $I$, and
\begin{align} \label{eq:binary}
 \bigg\| 2^N \sum_{i \in I_b} g_i \otimes g_i - \sum_{i \in I} g_i \otimes g_i \bigg\| \leq c_0 \sqrt{2^N \delta} \quad \text{for all} \quad b \in \{0,1\}^N
\end{align}
for an absolute constant $c_0>1$.
\end{theorem}

\section{Reproducing kernels on metric measure spaces} \label{sec:rkhs}
This section describes the precise setting of reproducing kernels on metric measure spaces to be used in this paper. It also collects various preliminary results that will be used throughout the paper.

\subsection{Assumptions} We first describe the precise assumptions on the metric measure space and reproducing kernels.

\subsection*{Metric measure space} Let $(X, d, \mu)$ be a metric measure space, consisting of a set $X$ equipped with a metric $d$ and a measure $\mu$. Throughout, we will assume that $\mu(X) = \infty$, and that all balls $B_r (x) = \{ y \in X : d(y,x) < r \}$ satisfy $\mu(B_r(x)) < \infty$ for all $r > 0$ and $x \in X$. In addition, we make the following standing assumptions:
\begin{enumerate}[-]
 \item \emph{Nondegeneracy of balls}: There exists $r>0$ such that
 \begin{align}\label{eq:NDB}
  \inf_{x \in X} \mu(B_r(x)) > 0.
 \end{align}
\item \emph{Weak annular decay property}:
\begin{align} \label{eq:wad}
 \lim_{r\to \infty} \sup_{x \in X} \frac{\mu(B_{r+1} (x) \setminus B_r (x))}{\mu(B_r(x))} = 0.
\end{align}
\item \emph{Doubling at large scale}: There exists $r_0 > 0$ and  $C_d > 0$ such that, for all $r\geq r_0$,
\begin{align} \label{eq:doubling}
 \mu(B_{2r}(x)) \leq C_d \mu(B_r (x)) \quad \text{for all} \quad x \in X.
\end{align}
\end{enumerate}
These standing assumptions are the same as in \cite{bownik2025redundancy}; see also \cite{fuehr2017density, mitkovsi2020density} for similar assumptions.

We will use two consequences of the weak annular decay property \eqref{eq:wad}. First, by \cite[Section 3]{fuehr2017density}, for all $r' > 0$, we have
 \begin{equation}\label{wap}
  \lim_{r \to \infty} \sup_{x \in X} \frac{\mu(B_{r+r'}(x))}{\mu(B_r(x))} = 1
 \quad \text{and} \quad
  \lim_{r \to \infty} \sup_{x \in X} \frac{\mu(B_{r+r'} (x) \setminus B_{r-r'} (x))}{\mu(B_r(x))} = 0.
 \end{equation}
Second, we have that
\begin{equation}\label{gi}
 \lim_{r \to \infty} \inf_{x \in X} \mu (B_r (x)) = \infty,
\end{equation}
cf. \cite[Lemma 3.2]{bownik2025redundancy}.

\subsection*{Reproducing kernel}
Throughout, we will denote by $\mathcal{H} \subseteq L^2 (X, \mu)$ an infinite-dimensional \emph{reproducing kernel Hilbert space (RKHS)}, i.e., a closed subspace such that, for each $x \in X$, the point evaluation $\mathcal{H} \ni f \mapsto f(x) \in \mathbb{C}$ is a well-defined bounded linear functional. Then, for each $x \in X$, there exists $k_x \in \mathcal{H}$ such that $f(x) = \langle f, k_x \rangle $ for all $f \in \mathcal{H}$. The function $k_x$ is said to be the \emph{reproducing kernel at $x$}. The \emph{reproducing kernel for $\mathcal{H}$} is defined as
\[
 k : X \times X \to \mathbb{C}, \; k(x,y) = \langle k_y, k_x \rangle.
\]
Note  $k_x(y) = \overline{k(x,y)} = k(y,x)$, and that each $f \in \mathcal{H}$ satisfies the reproducing formula
\[
 f(x) = \langle f, k_x \rangle = \int_X  f(y) k(x,y) \; d\mu(y), \quad x \in X.
\]
In particular, the system $\{k_x \}_{x \in X}$ is a continuous Parseval frame for $\mathcal{H} \subseteq L^2 (X, \mu)$.

A condition on the reproducing kernel that we will always assume is the following:
\begin{enumerate}[-]
 \item \emph{Diagonal condition}: There exist constants $0 < C_1 \leq C_2 < \infty$ such that
 \begin{align} \label{eq:dc}
  C_1 \leq k(x,x) \leq C_2
 \end{align}
for all $x \in X$.
\end{enumerate}

\subsection{Beurling densities}
The \emph{lower Beurling density} and \emph{upper Beurling density} of a set $\Lambda \subseteq X$ are defined as
\begin{align} \label{eq:lower_density}
 D^-(\Lambda) = \liminf_{r \to \infty} \inf_{x \in X} \frac{\#(\Lambda \cap B_r (x))}{\mu(B_r(x))}
\end{align}
and
\begin{align} \label{eq:upper_density}
 D^+(\Lambda) = \limsup_{r \to \infty} \sup_{x \in X} \frac{\#(\Lambda \cap B_r (x))}{\mu(B_r(x))},
\end{align}
respectively. We also occasionally write $D^{\pm}_{\mu}$ to stress the dependence of the densities on the measure $\mu$. In the particular case of the measure $\mu_0$ given by $d\mu_0 (y) := k(y,y) d\mu(y)$ for $y \in X$, we denote the associated density often simply by $D^{\pm}_{0} = D^{\pm}_{\mu_0}$. Note that the weighted measure $\mu_0$ satisfies the standing assumptions by the diagonal condition \eqref{eq:dc}.

We treat a countable set $\Lambda \subseteq X$ as an index family and allow for multiplicities.

We need several lemmas about Beurling densities. The first one is \cite[Lemma 3.2]{bownik2025redundancy}.

\begin{lemma} [\cite{bownik2025redundancy}] \label{lem:partition}
\label{exist}
 Let $\Lambda \subseteq X$ be such that $0<D^-_{\mu}(\Lambda) \le D^+_{\mu}(\Lambda)<\infty$. 
Then, there exists a partition $\{X_y\}_{y\in Y}$ of $X$ satisfying for some $R>0$ and $C>1$,
 \begin{align}\label{xy}
 B_{R}(y) & \subseteq X_y \subseteq B_{3R}(y) \qquad \text{for all }y\in Y,
\\
\label{reg}
 \frac{1}{C} D^-_{\mu}(\Lambda) & \leq \frac{\#(\Lambda \cap X_y)}{\mu(X_y)} \leq C D^{+}_{\mu}(\Lambda) \qquad \text{for all } y \in Y.
\end{align}
\end{lemma}

We will use the following property of the partitions in Lemma~\ref{lem:partition}.

\begin{lemma}\label{den}
 Let $\Lambda \subseteq X$. If $\{X_y\}_{y\in Y}$ is  a partition of $X$ satisfying \eqref{xy} for some $R>0$, then,
\begin{equation}\label{den0}
\inf_{y\in Y} \frac{\#(\Lambda \cap X_y)}{\mu(X_y)} 
\le D^-_{\mu}(\Lambda) \le D^+_{\mu}(\Lambda) \le
\sup_{y\in Y} \frac{\#(\Lambda \cap X_y)}{\mu(X_y)}.
\end{equation}
\end{lemma}

\begin{proof}
We will follow an argument as in \cite[Lemma 3.6]{bownik2025redundancy}.
By \eqref{xy}, we get for $x \in X$ and $r > 6R$,
\begin{align} \label{eq:nested1}
 B_r(x) \subseteq \bigcup_{y \in Y :\  X_y \cap B_{r} (x) \ne \emptyset} X_y \subseteq B_{r+6R} (x)
\end{align}
and
\begin{align} \label{eq:nested2}
 B_{r-6R} (x) \subseteq \bigcup_{y \in Y :\  X_y \subseteq B_{r} (x)} X_y \subseteq B_{r} (x).
\end{align}
Using \eqref{eq:nested1}, we thus get
\begin{align*}
 \frac{\#(\Lambda \cap B_r (x))}{\mu(B_r(x))}
  &\leq 
 \frac{\mu( B_{r+6R}(x))}{\mu(B_r(x))}  \frac{\#\big(\Lambda \cap \bigcup_{y \in Y : X_y \cap B_{r} (x) \ne \emptyset} X_y \big)}{\mu \big(\bigcup_{y \in Y: X_y \cap B_{r} (x) \ne \emptyset} X_y\big)}  \\
& \le  \frac{\mu( B_{r+6R}(x))}{\mu(B_r(x))} \sup_{y\in Y} \frac{\#(\Lambda \cap X_y)}{\mu(X_y)}.
\end{align*}
By \eqref{wap}, we have $\lim_{r \to \infty} \sup_{x \in X} \frac{\mu( B_{r+6R}(x))}{\mu(B_r(x))} = 1$. Hence, the upper bound in \eqref{den0} follows. To show the lower bound we use the fact that $\lim_{r \to \infty} \inf_{x \in X} \frac{\mu( B_{r-6R}(x))}{\mu(B_r(x))} = 1$ and
\begin{align*}
 \frac{\#(\Lambda \cap B_r (x))}{\mu(B_r(x))}
  &\geq 
 \frac{\mu( B_{r-6R}(x))}{\mu(B_r(x))}  \frac{\#\big(\Lambda \cap \bigcup_{y \in Y : X_y \subseteq B_{r} (x) } X_y \big)}{\mu \big(\bigcup_{y \in Y: X_y \subseteq B_{r} (x)} X_y\big)}  \\
& \ge  \frac{\mu( B_{r-6R}(x))}{\mu(B_r(x))} \inf_{y\in Y} \frac{\#(\Lambda \cap X_y)}{\mu(X_y)}.
\end{align*}
This completes the proof.
\end{proof}

We will also use the following regularization lemma, cf. \cite[Lemma 3.5]{bownik2025redundancy}.

\begin{lemma}[\cite{bownik2025redundancy}] \label{regularization}
 Let $\Lambda \subseteq X$ be such that $0<D^-_{\mu}(\Lambda) \le D^+_{\mu}(\Lambda)<\infty$. 
Let $\{X_y\}_{y\in Y}$ be a partition of $X$ satisfying \eqref{xy} and \eqref{reg} for some $R>0$ and $C>1$.
Let $\ve>0$. 
  Then, for sufficiently large $R'>R$, there exists a partition $\{X'_{y}\}_{y \in Y'}$ of $X$ satisfying:
 \begin{align}\label{xyh}
 B_{R'}(y) & \subseteq X'_{y} \subseteq B_{3R'}(y) \qquad \text{for all }y\in Y',
 \\
 \label{reg0}
 D^-_{\mu}(\Lambda) - \ve &\le \frac{\#(\Lambda \cap X'_{y})}{\mu(X'_{y})} \le D^{+}_{\mu}(\Lambda) + \ve \qquad \text{for all }y\in Y',
 \\
 \label{nest}
&  \forall y \in Y\quad  \exists y' \in Y' \quad X_y \subseteq X'_{y'}.
\end{align} 
 \end{lemma}

Lastly, we prove the following result on the existence of adequate binary selectors.

\begin{proposition}\label{geo}
Let $\Lambda \subseteq X$ be such that $0<D^-_{\mu}(\Lambda) \le D^+_{\mu}(\Lambda)<\infty$. Let $N\in \N$ and $\ve>0$. Then, there exists a choice of consecutive partitions of $\Lambda$ into sets of size $\le 2$, such that every binary selector $I_b$, $b\in \{0,1\}^N$, of order $N$ satisfies
\begin{equation}\label{geo1}
\big| 2^N D^-_\mu(I_b) - D^-_\mu(\Lambda) \big| \le \ve.
\end{equation}
\end{proposition}

\begin{proof} 
By Lemma \ref{regularization}, there exists a partition $\{X_y\}_{y\in Y}$ and $R>0$  such that \eqref{xyh} and \eqref{reg0} hold. By taking $R>0$ sufficiently large, \eqref{gi} guarantees that
\begin{equation}\label{big}
\mu( X_{y}) > 2^N/\ve.
\end{equation}
for $y \in Y$.

Let $\{J_k\}_{k\in K}$ be a partition of $\Lambda$ into sets of size $\le 2$ such that:
\begin{itemize}
\item for all $k\in K$, there exists $y\in Y$ such that $J_k \subset X_y$,
\item if $y\in Y$, then for all, but at most one, $k\in K$, such that $J_k \subset X_y$, we have $\#(J_k \cap X_y) =2$.
\end{itemize}
The last property is equivalent with the following. If $\#(\Lambda \cap X_y)$ is even, then we require that all $J_k \subset X_y$ have size 2. If $\#(\Lambda \cap X_y)$ is odd, then there is exactly one $k\in K$ such that $J_k \subset X_y$ and $J_k$ is a singleton. Let $I_0$ and $I_1$ be any binary selectors of order $1$. Then, for any $y\in Y$ we have
\[
|2 \#(I_i \cap X_y) - \#(\Lambda \cap X_y)| \le 1 \qquad \text{for }i=0,1.
\]

We construct consecutive partitions of $\Lambda$ into sets of size $\le 2$ as follows. Suppose that for some $j\ge 1$, binary selectors $I_b$,  $b\in \{0,1\}^{j}$, of order $j$ are already defined. Given $b\in \{0,1\}^j$, we let $\{J_k\}_{k\in K}$ be a partition of $I_b$ into sets of size $\le 2$ such that:
\begin{itemize}
\item for all $k\in K$, there exists $y\in Y$ such that $J_k \subset X_y \cap I_b$,
\item if $y\in Y$, then for all, but at most one, $k\in K$, such that $J_k \subset X_y\cap I_b$, we have $\#(J_k \cap X_y) =2$.
\end{itemize}
That is, all $k\in K$ such that $J_k \subset X_y\cap I_b$, satisfy $\#(J_k \cap X_y) =2$, whenever $\#(X_y\cap I_b)$ is even. If $\#(X_y\cap I_b)$ is odd, there is exactly one exception when $J_k$ is a singleton. Let $I_{b0}$ and $I_{b1}$ be any binary selectors of order $1$. 
Then, for any $y\in Y$ we have
\begin{equation}\label{geo5}
|2 \#(I_{bi} \cap X_y) - \#(I_b \cap X_y)| \le 1 \qquad \text{for }i=0,1.
\end{equation}

Consequently, we obtain binary selectors $I_b$, $b\in \{0,1\}^N$, of order $N\in \N$, satisfying
\begin{equation}\label{geo10}
|2^N \#(I_b \cap X_y) - \#(\Lambda \cap X_y) | \le 2^N-1 \qquad\text{for all }y\in Y.
\end{equation}
Indeed, fix $b\in \{0,1\}^N$. For any $0 \le j \le N$, let $b_j$ be the first $j$ components of $b$. By telescoping and \eqref{geo5} we have
\[
\begin{aligned}
|2^N \#(I_b \cap X_y) - \#(\Lambda \cap X_y) | & \le \sum_{j=0}^{N-1} |2^{j+1} \#(I_{b_{j+1}} \cap X_y) - 2^j \#(I_{b_j} \cap X_y) |
\\
& 
=  \sum_{j=0}^{N-1} 2^j |2 \#(I_{b_{j+1}} \cap X_y) - \#(I_{b_j} \cap X_y) |  \\
&\le \sum_{j=0}^{N-1} 2^j =2^N-1.
\end{aligned}
\]
Hence, by \eqref{big}, we have
\begin{equation}\label{il}
\frac{|2^N \#(I_b \cap X_y) - \#(\Lambda \cap X_y) |}{\mu(X_y)} < \ve
\end{equation}
for any $b\in \{0,1\}^N$.

By Lemma \ref{regularization}, there exists a partition $\{X'_{y}\}_{y \in Y'}$ of $X$ satisfying \eqref{xyh}, \eqref{reg0}, \eqref{nest}, and such that for all $b\in \{0,1\}^N$ we have
\begin{equation}\label{reg9}
 D^-_{\mu}(I_b) - 2^{-N} \ve \le \frac{\#(I_b \cap X'_{y})}{\mu(X'_{y})} \le D^{+}_{\mu}(I_b) + 2^{-N} \ve \qquad \text{for all }y\in Y'.
\end{equation}
Indeed, this follows by the iterative application of Lemma \ref{regularization} for each set $I_b$ to increasingly coarser partitions of $X$, due to the nested property \eqref{nest}. At each application of Lemma \ref{regularization} we take the previously constructed partition $\{X'_{y}\}_{y \in Y'}$, which satisfies \eqref{reg9} for some finite subcollection of $ b\in \{0,1\}^N$, to produce a coarser partition $\{X''_{y}\}_{y \in Y''}$, which satisfies \eqref{reg9} for an additional element $b\in \{0,1\}^N$. After a finite number of steps we exhaust all elements and obtain a final partition, also denoted by $\{X'_{y}\}_{y \in Y'}$, that satisfies  \eqref{reg9} for all $b\in \{0,1\}^N$.

Combining Lemma \ref{den} and \eqref{reg0} yields
\[
D^-_{\mu}(\Lambda) - \ve \le \inf_{y\in Y'} \frac{\#(\Lambda \cap X'_{y})}{\mu(X'_{y})} \le D^-_{\mu}(\Lambda) .
\]
Likewise, combining Lemma \ref{den} and \eqref{reg9} yields
\[
2^N D^-_{\mu}(I_b) - \ve \le \inf_{y\in Y'} \frac{2^N \#(I_b\cap X'_{y})}{\mu(X'_{y})} \le 2^N D^-_{\mu}(I_b) .
\]
On the other hand, the nested property and \eqref{il} also imply that
\[
\frac{|2^N \#(I_b \cap X'_y) - \#(\Lambda \cap X'_y) |}{\mu(X'_y)} < \ve \qquad\text{for all }y\in Y'.
\]
Therefore, we deduce that
\[
\big| 2^N D^-_\mu(I_b) - D^-_\mu(\Lambda) \big| \le 2\ve,
\]
which settles the claim.
\end{proof}

\subsection{Localization conditions}
In addition to the assumptions on the metric measure space and the reproducing kernel, we will also consider various conditions on systems of vectors in the reproducing kernel Hilbert space. Explicitly, given a system of vectors $\{g_x\}_{x \in X}$ in the reproducing kernel Hilbert space $\mathcal{H} \subseteq L^2 (X)$, we consider the conditions:
\begin{enumerate}[-]
\item \emph{Weak localization property}: For every $\varepsilon > 0$, there exists $r = r(\varepsilon) > 0$ such that
\begin{align} \label{eq:wl}
 \sup_{x \in X} \int_{X \setminus B_r (x)} |g_x(y)|^2 \; d\mu(y) < \varepsilon^2.
\end{align}
\item \emph{Homogeneous approximation property}: If $\Lambda \subseteq X$ is a countable set such that $\{ g_{\lambda} \}_{\lambda \in \Lambda}$ is a Bessel sequence in $\mathcal{H}$, then for every $\varepsilon > 0$, there exists $r = r(\varepsilon) > 0$ such that
\begin{align} \label{eq:hap}
 \sup_{x \in X} \sum_{\lambda \in \Lambda \setminus B_r (x)} |\gf (x)|^2 < \varepsilon^2.
\end{align}
\end{enumerate}

In contrast to the standing assumptions on the metric measure space and reproducing kernel, we will not always assume the conditions \eqref{eq:wl} and \eqref{eq:hap} to hold. In fact, the minimal assumption required for the present paper is that the vectors $\{g_x\}_{x \in X}$ satisfy the so-called \emph{frame redundancy/density property} (cf. \cite[Definition 4.1]{bownik2025redundancy}), which we recall next.

If $\Lambda \subseteq X$ is countable and $\{\gf\}_{\lambda \in \Lambda}$ is a frame for $\mathcal{H}$ with canonical dual frame $\{\df\}_{\lambda \in \Lambda}$, then we define the associated \emph{lower frame measure} and \emph{upper frame measure} by
\begin{align*}
 M^-(\{\gf \}_{\lambda \in \Lambda})  := \liminf_{r \to \infty} \inf_{x \in X} \frac{1}{\#(\Lambda \cap B_r (x))} \sum_{\lambda \in \Lambda \cap B_r (x)} \langle \gf, \df \rangle
\end{align*}
and
\begin{align*}
 M^+(\{\gf \}_{\lambda \in \Lambda})  := \limsup_{r \to \infty} \sup_{x \in X} \frac{1}{\#(\Lambda \cap B_r (x))} \sum_{\lambda \in \Lambda \cap B_r (x)} \langle \gf, \df \rangle,
\end{align*}
 respectively. The following definition corresponds to \cite[Definition 4.1]{bownik2025redundancy}.

\begin{definition}
A system of vectors  $\{g_x \}_{x \in X}$ is said to satisfy the \emph{frame measure/density property} if for each countable set $\Lambda \subseteq X$ such that $\{\gf \}_{\lambda \in \Lambda}$ is a frame for $\mathcal{H}$ the following formulae hold:
  \begin{align} \label{eq:frd}
  M^-(\{\gf \}_{\lambda \in \Lambda}) = \frac{1}{D_0^+ (\Lambda)} \quad \text{and} \quad M^+(\{\gf \}_{\lambda \in \Lambda}) = \frac{1}{D_0^- (\Lambda)}.
 \end{align}
\end{definition}

The following theorem is \cite[Theorem 4.3]{bownik2025redundancy}; see also \cite{fuehr2017density, mitkovsi2020density, olevskii2012revisiting, balan2006density2} for closely related results in various settings.

\begin{theorem} \label{thm:framemeasure}
 Suppose that $\{g_x \}_{x \in X}$ is a system of vectors satisfying $\inf_{x \in X} \| g_x \|^2 > 0$, the weak localization condition \eqref{eq:wl} and the homogeneous approximation property \eqref{eq:hap}. Then $\{g_x \}_{x \in X}$ satisfies the frame measure/density property \eqref{eq:frd}.
\end{theorem}

\section{Riesz sequences near critical density} \label{sec:main}
Throughout this section, we let $\mathcal{H} \subseteq L^2 (X)$ be a reproducing kernel Hilbert space satisfying the assumption of Section~\ref{sec:rkhs}.

We start with the following lemma.

\begin{lemma} \label{lem:upper_lambda_a}
Let $\{g_x \}_{x \in X}$ be a system of vectors in $\mathcal{H}$ satisfying the frame measure/density property \eqref{eq:frd}.

Suppose that $\Lambda \subseteq X$ is a countable set such that $\{g_{\lambda} \}_{\lambda \in \Lambda}$ is a frame for $\mathcal{H}$ with canonical dual frame $\{h_{\lambda} \}_{\lambda \in \Lambda}$ and that $0<D_0^-(\Lambda) \leq D^+_0(\Lambda) < \infty$.  For $0< \alpha < 1$, define  $\Lambda_{\alpha} := \{ \lambda \in \Lambda : \langle g_{\lambda}, h_{\lambda} \rangle \leq \alpha \big\}$. Then
\begin{align} \label{eq:upper_lambda_a}
 D^+_0 (\Lambda_{\alpha}) \leq \frac{D^+_0 (\Lambda) - 1}{1-\alpha}
\end{align}
\end{lemma}
\begin{proof}
Under the assumptions, we have that
$
 M^- (\{g_{\lambda} \}_{\lambda \in \Lambda}) = \frac{1}{D^+_0 (\Lambda)}.
$

Choose $r_n \in (0, \infty)$ such that $r_n \to \infty$ as $n \to \infty$ and \(x_n\in X\) such that
\[
D_0^+(\Lambda_\alpha)
=
\lim_{n\to\infty}
\frac{\#\bigl(\Lambda_\alpha\cap B_{r_n}(x_n)\bigr)}
{\displaystyle\int_{B_{r_n}(x_n)}k(y,y)\,d\mu(y)},
\]
and define
\[
\widetilde M^-(\{g_\lambda\}_{\lambda \in \Lambda})
:=
\liminf_{n\to\infty}
\frac{1}{\#\bigl(\Lambda\cap B_{r_n}(x_n)\bigr)}
\sum_{\lambda\in\Lambda\cap B_{r_n}(x_n)}
\langle g_\lambda,h_\lambda\rangle.
\]
Note that
$
0\leq  M^-(\{g_\lambda\}_{\lambda \in \Lambda})
\leq\widetilde M^-(\{g_\lambda\}_{\lambda \in \Lambda})
\leq1.
$

Fix \(\varepsilon>0\). Choose subsequences \((r_{n_k})_{k \in \mathbb{N}} \) and
\((x_{n_k})_{k \in \mathbb{N}} \), and \(N\in\mathbb N\), such that
\[
\left|
\frac{1}{\#\bigl(\Lambda\cap B_{r_{n_k}}(x_{n_k})\bigr)}
\sum_{\lambda\in\Lambda\cap B_{r_{n_k}}(x_{n_k})}
\langle g_\lambda,h_\lambda\rangle
- \widetilde M^-(\{g_\lambda\}_{\lambda \in \Lambda})
\right|
<\varepsilon,
\qquad k\geq N.
\]
Then, for \(k\geq N\), setting
$
B_k:=B_{r_{n_k}}(x_{n_k}),
$
we have
\begin{align*}
\widetilde M^-(\{g_\lambda\}_{\lambda \in \Lambda}) - \varepsilon
&\leq
\frac{1}{\#\bigl(\Lambda\cap B_k \bigr)}
\sum_{\lambda\in\Lambda\cap B_k}
\langle g_\lambda,h_\lambda\rangle
\\
&=
\frac{1}{\#(\Lambda\cap B_k)}
\left(
\sum_{\lambda\in\Lambda_\alpha\cap B_k}
\langle g_\lambda,h_\lambda\rangle
+
\sum_{\lambda\in(\Lambda\setminus\Lambda_\alpha)\cap B_k}
\langle g_\lambda,h_\lambda\rangle
\right)
\\
&\leq
\frac{
\#\bigl((\Lambda\setminus\Lambda_\alpha)\cap B_k\bigr)
+\alpha\#(\Lambda_\alpha\cap B_k)}
{\#(\Lambda\cap B_k)}
\\
&=
\frac{
\#(\Lambda\cap B_k)
-(1-\alpha)\#(\Lambda_\alpha\cap B_k)}
{\#(\Lambda\cap B_k)}.
\end{align*}
This easily yields 
\[
\frac{\#(\Lambda_\alpha\cap B_k)}{\int_{B_k} k(y,y) \; d\mu(y)}
\leq
\frac{1-\widetilde M^-(\{g_\lambda\}_{\lambda \in \Lambda})+\varepsilon}{1-\alpha}
\cdot \frac{\#(\Lambda\cap B_k)}{\int_{B_k} k(y,y) \; d\mu(y)}.
\]
Hence, letting \(k\to\infty\), we obtain
\begin{align*}
D_0^+(\Lambda_\alpha)
&\leq
\frac{1- \widetilde M^-(\{g_\lambda\}_{\lambda \in \Lambda}) +\varepsilon}{1-\alpha}
\limsup_{k\to\infty}
\frac{\#(\Lambda\cap B_k)}{\int_{B_k} k(y,y) \; d\mu(y)}
\\
&\leq
\frac{1- \widetilde M^-(\{g_\lambda\}_{\lambda \in \Lambda}) +\varepsilon}{1-\alpha}
D_0^+(\Lambda)
\\
&\leq
\frac{1- M^-(\{g_\lambda\}_{\lambda \in \Lambda}) +\varepsilon}{1-\alpha}
D_0^+(\Lambda).
\end{align*}
Since $\varepsilon > 0$ was chosen arbitrary, this yields the result.
\end{proof}

The following result is the main theorem of this paper.

\begin{theorem} \label{thm:main}
Suppose $\{g_x \}_{x \in X}$ is a frame for $\mathcal{H} \subseteq L^2 (X)$ satisfying the frame measure/density property \eqref{eq:frd} and such that
\[
 0 < c \leq \| g _x \|^2 \leq C < \infty \quad \text{for all} \quad x \in X.
\]
Then, for every $\varepsilon > 0$, there exists a countable set $\Lambda \subseteq X$ satisfying
 $
  D^-_0 (\Lambda) \geq 1 - \varepsilon
 $
 and such that $\{g_{\lambda} \}_{\lambda \in \Lambda}$ is a Riesz sequence in $\mathcal{H}$.

 In addition, if $\{g_x \}_{x \in X}$ is a Parseval frame for $\mathcal{H}$ satisfying $C := \| g_x \|^2$ for all $x \in X$, then the Riesz bounds of the Riesz sequence $\{g_{\lambda} \}_{\lambda \in \Lambda}$ can be chosen to be of the form $ c_1 (\ve)C$ and $c_2 C$ for constants $c_1(\ve), c_2>0$ with $c_1(\ve)$ depending only on $\ve$.
\end{theorem}
\begin{proof}
Given $\ve \in (0,1)$, we define $\alpha=\frac{\ve}{16c_0^2}$, $\ve'=\frac{\ve^2}{32c_0^2}$,
where $c_0>0$ is the absolute constant from Theorem \ref{thm:binary}.
  Under the assumptions, an application of \cite[Theorem 5.4]{bownik2025redundancy} implies that  there exists $\Gamma \subseteq X$ satisfying $D^+_0 (\Gamma) \leq 1+\varepsilon'$ and such that $\{g_{\gamma} \}_{\gamma \in \Gamma}$ is a frame for $\mathcal{H}$.
 Denote by $S_{\Gamma} : \mathcal{H} \to \mathcal{H}$ the frame operator of $\{g_{\gamma} \}_{\gamma \in \Gamma}$. Let $\{h_{\gamma} \}_{\gamma \in \Gamma} = \big\{S^{-1/2}_{\Gamma} g_{\gamma} \big\}_{\gamma \in \Gamma}$ be the canonical Parseval frame of $\{g_{\gamma} \}_{\gamma \in \Gamma}$. For $\alpha \in (0,1)$, consider
 \[
  \Gamma_{1-\alpha} = \big\{ \gamma \in \Gamma : \| h_{\gamma} \|^2 \leq 1-\alpha \big\}
 \]
and set $\Gamma' := \Gamma \setminus \Gamma_{1-\alpha}$. By Lemma \ref{lem:upper_lambda_a} we have
\[
 D^+_0 (\Gamma_{1-\alpha}) \leq \frac{D^+_0 (\Gamma) - 1}{\alpha} \le \frac{\ve'}{\alpha} = \frac{\ve}2.
 \]
On the other hand, the necessary density condition for the frame $\{g_{\gamma} \}_{\gamma \in \Gamma}$ yields that $D^-_0(\Gamma) \ge 1$. Thus,
\begin{equation}\label{dga}
D^-_0(\Gamma')  \ge D^-_0(\Gamma) - D^+_0(\Gamma_{1-\alpha}) \ge 1- \frac{\ve}2.
\end{equation}

Note that $\{h_{\gamma} \}_{\gamma \in \Gamma'}$ is a Bessel sequence in $\mathcal{H}$ with Bessel bound $1$ and satisfies $\| h_{\gamma} \|^2 > 1 - \alpha$ for all $\gamma \in \Gamma'$. By \Cref{lem:bessel_completion}, there exists a Hilbert space $\widetilde{\mathcal{H}} \supseteq \mathcal{H}$ and vectors $\varphi_i \in \widetilde{\mathcal{H}}$, $i \in \mathbb{N}$, satisfying $\| \varphi_n \|^2 = 1- \alpha$ for all $n \in \mathbb{N}$ and such that $\{h_{\gamma} \}_{\gamma \in \Gamma'} \cup \{\varphi_n \}_{n \in \mathbb{N}}$ is a Parseval frame for $\widetilde{\mathcal{H}}$.
Set $I := \Gamma' \cupdot \mathbb{N}$ as a disjoint union of $\Gamma'$ and $\N$. 
By Naimark's dilation theorem (Lemma \ref{lem:naimark}), we can embed $\widetilde{\mathcal{H}}$ into $\ell^2(I)$ such that $h_{\gamma} = P e_{\gamma}$ for $\gamma \in \Gamma'$ and $\varphi_n = P e_n$ for $n \in \mathbb{N}$, where $(e_i)_{i \in I}$ is the standard basis for $\ell^2 (I)$ and $P$ is the orthogonal projection of $\ell^2 (I)$ onto $\widetilde{\mathcal{H}}$. For $\gamma \in \Gamma'$ and $i \in \mathbb{N}$, let $\widetilde{h_{\gamma}} := (\mathbf{I} - P) e_{\gamma}$ and $\widetilde{\varphi}_n = (\mathbf{I} - P) e_n$. Then, $\{\widetilde{h_{\gamma}} \}_{\gamma \in \Gamma'} \cup \{\widetilde{\varphi_n} \}_{n \in \mathbb{N}}$ is a Parseval frame for the orthogonal complement $\widetilde{\mathcal{H}}^{\perp}$ of $\widetilde{\mathcal{H}} \subset \ell^2(I)$, which is known as the Naimark complement of $\{h_{\gamma} \}_{\gamma \in \Gamma'} \cup \{\varphi_n \}_{n \in \mathbb{N}}$. Note that
\begin{equation}\label{ns}
\| \widetilde{h_{\gamma}} \|^2, \| \widetilde{\varphi_{n}} \|^2 \leq \alpha
\qquad\text{for all } \gamma \in \Gamma',  n \in \mathbb{N}.
\end{equation}

The main goal is to combine Lemma \ref{lem:naimark}, Theorem \ref{thm:binary}, and Proposition \ref{geo} as follows.
By Lemma \ref{lem:naimark}, for any $J \subseteq I$ and $\delta > 0$, the system $\{ h_{\gamma} \}_{\gamma \in \Gamma' \cap J} \cup \{ \varphi_{i} \}_{i \in \mathbb{N} \cap J}$ is a Riesz sequence in $\widetilde{\mathcal{H}}$ with lower bound $\delta$ if and only if the system $\{ \widetilde{h_{\gamma}} \}_{\gamma \in \Gamma' \cap J} \cup \{ \widetilde{\varphi_{n}} \}_{n \in \mathbb{N} \cap J}$ is a Bessel sequence in $\widetilde{\mathcal{H}}$ with Bessel bound $1-\delta$. Consequently, if we can control the Bessel bound of the appropriately selected part $\{ \widetilde{h_{\gamma}} \}_{\gamma \in \Gamma' \cap J} \cup \{ \widetilde{\varphi_{n}} \}_{n \in \mathbb{N} \cap J}$ of the Naimark complement, then the collection $\{ h_{\gamma} \}_{\gamma \in \Lambda}$ is a Riesz sequence, where $\Lambda= \Gamma' \cap J$. To achieve this we will use Theorem \ref{thm:binary} to construct an appropriate binary selector $I_b \subseteq I$ and let $J=I \setminus I_b$, while Lemma \ref{geo} will guarantee that the Beurling density of $\Lambda$ stays close to $1$. By doing this we follow the general strategy from \cite[Corollary 6.4]{bownik2024selector}.

To wit, we let $N\in \N$ be such that
\begin{equation}\label{N}
\frac{2}{\varepsilon}= \frac{1}{8c_0^2\alpha} \le 2^N \le \frac{1}{4c_0^2\alpha}=\frac{4}{\varepsilon} .
\end{equation}
By Proposition \ref{geo}, for any $\ve''>0$,  there exists a choice of consecutive partitions of $\Gamma'$ into sets of size $\le 2$, such that every binary selector $I'_b$, $b\in \{0,1\}^N$, of order $N$ satisfies
\begin{equation}\label{g1}
\big| 2^N D^-_0(I'_b) - D^-_0(\Gamma') \big| \le \ve''.
\end{equation}
Take $\ve'' :=\ve^2/4$. Then, by \eqref{dga}, \eqref{N}, and \eqref{g1} we have
\begin{equation}\label{g2}
\begin{aligned}
D^-_0(\Gamma' \setminus I'_b)  &\ge \sum_{b'\in\{0,1\}^N: \ b' \ne b} D^-_0(I'_{b'}) \ge  (1-2^{-N}) D^-_0(\Gamma') - \ve''
\\
&\ge (1-\ve/2)^2 - \ve'' = 1-\ve.
\end{aligned}
\end{equation}

Next we apply Theorem \ref{thm:binary} to the Parseval frame $\{\widetilde{h_{\gamma}} \}_{\gamma \in \Gamma'} \cup \{\widetilde{\varphi_n} \}_{n \in \mathbb{N}}$ for $\widetilde{\mathcal{H}}^{\perp}$ satisfying \eqref{ns}. We extend the above  consecutive partitions of $\Gamma'$ into a partition of $I = \Gamma' \cupdot \mathbb{N}$ by partitioning $\N$ arbitrarily into sets of size $2$. Theorem  \ref{thm:binary} guarantees existence of binary selectors $I_b$, $b \in \{0,1\}^N$, that form a partition of $I$, such that for all $ b \in \{0,1\}^N$
\[
 \bigg\| 2^N \bigg(\sum_{\gamma \in \Gamma' \cap I_b} \widetilde{h_{\gamma}} \otimes \widetilde{h_{\gamma}} + \sum_{n \in \N \cap I_b} \widetilde{\varphi_n} \otimes \widetilde{\varphi}_n \bigg) - \mathbf{I}_{\widetilde{\mathcal{H}}^{\perp}}  \bigg\| \leq c_0 \sqrt{2^N \alpha} \le 1/2.
\]
Thus,
\[
\bigg(\sum_{\gamma \in \Gamma' \cap I_b} \widetilde{h_{\gamma}} \otimes \widetilde{h_{\gamma}} + \sum_{n \in \N \cap I_b} \widetilde{\varphi}_n \otimes \widetilde{\varphi}_n \bigg) \ge 2^{-N-1}\mathbf{I}_{\widetilde{\mathcal{H}}^{\perp}} \ge \frac{\ve}{8} \mathbf{I}_{\widetilde{\mathcal{H}}^{\perp}}.
\]
Hence, the system $\{\widetilde{h_{\gamma}} \}_{\gamma \in \Gamma' \cap I_b} \cup \{\widetilde{\varphi_n} \}_{n \in \mathbb{N} \cap I_b}$ is a frame for $\widetilde{\mathcal{H}}^{\perp}$ with lower frame bound $ \ve/8$ and upper frame bound $1$. Letting $J := I \setminus I_b$, this implies that complementary system $\{\widetilde{h_{\gamma}} \}_{\gamma \in \Gamma' \cap J} \cup \{\widetilde{\varphi_n} \}_{n \in \mathbb{N} \cap J}$ is a Bessel sequence in $\widetilde{\mathcal{H}}^{\perp}$ with bound $1-\ve/8$. By Lemma \ref{lem:naimark}, we conclude that $\{ h_{\gamma} \}_{\gamma \in \Gamma' \cap J} \cup \{\varphi_n \}_{n \in \mathbb{N} \cap J}$ is a Riesz sequence in $\widetilde{\mathcal{H}}$ with Riesz bounds $\ve/8$ and $1$.
Therefore, it follows that $\{ h_{\gamma} \}_{\gamma \in \Gamma' \cap J}$ is also a Riesz sequence in $\widetilde{\mathcal{H}}$ with Riesz bounds $\ve/8$ and $1$, and hence so it is in $\mathcal{H}$. Set $\Lambda := \Gamma' \cap J = \Gamma' \setminus (\Gamma' \cap I_b)$. Then, since $S_\Gamma$ is invertible, we have that
$ \{ g_{\gamma} \}_{\gamma \in \Lambda} = \{S^{1/2}_{\Gamma} h_{\gamma} \}_{\gamma \in \Lambda}$ is also a Riesz sequence in $\mathcal{H}$.
Lastly, since $I'_b := \Gamma' \cap I_b$ is a binary selector for the original consecutive partitions of $\Gamma'$ and $\Lambda = \Gamma' \setminus  I'_b$, it follows by \eqref{g2} that $D_0^-(\Lambda) \ge 1-\ve$.

For the additional part, we note that if $\{g_x \}_{x \in X}$ is a Parseval frame for $\mathcal{H}$ with $C := \| g_x \|^2$ for all $x \in X$, then the additional statement in \cite[Theorem 5.4]{bownik2025redundancy} yields that the frame $\{g_{\gamma} \}_{\gamma \in \Gamma}$ has frame bounds $c_1 (\ve') C$ and $c_2 C$. Therefore, the Riesz sequence $\{g_{\lambda} \}_{\lambda \in \Lambda}$ constructed above has Riesz bounds $\ve c_1(\ve') C/8$ and $c_2 C$, because we have the estimates
$\ve/(8 c_1(\ve')) C \leq \ve/(8\|S_{\Gamma}^{-1}\|)$ and
$\|S_{\Gamma} \| \leq c_2 C. $
This completes the proof.
\end{proof}

\section{Examples} \label{sec:examples}
In this section, we discuss various examples to which the main results of this paper apply and yield new results. We refer to \cite[Section 6]{bownik2025redundancy} for a more extensive discussion of these examples; see also \cite[Section 5]{fuehr2017density} and \cite[Section 6]{mitkovsi2020density} for various other examples.

\subsection{Euclidean space}
Throughout this subsection, we consider the classical case of $X$ being some Euclidean space, $d$ Euclidean distance on $X$ and $\mu$ Lebesgue measure on $X$.  It is readily verified that  our standing assumptions \eqref{eq:NDB}, \eqref{eq:wad} and \eqref{eq:doubling} are satisfied for the metric measure space $(X, d, \mu)$. We consider three classes of examples.

\subsubsection*{Exponential systems} For a set $\Omega \subseteq \R^d$ of finite measure, the associated Paley-Wiener space $\PW_{\Omega} $ consist of all  $f \in L^2 (\R^d)$ whose Fourier transform $(\mathcal{F} f) (\xi) =\int_{\R^d} f(t) e^{2\pi i \xi \cdot t} \; dt$ has support inside $\Omega$. The space $\PW_{\Omega}$ is a closed subspace of $L^2 (\R^d)$ with reproducing kernels given by
\[
k_x(y) = (\mathcal{F}^{-1} \mathds{1}_{\Omega}) (y-x), \quad x, y \in \R^d.
\]
In particular, we have $\| k_x \|^2 = |\Omega|$ for all $x \in \R^d$.
It is well-known that if $\Omega$ is bounded, then the reproducing kernels $\{k_x \}_{x \in \R^d}$ satisfy the homogeneous approximation property \eqref{eq:hap}, cf. \cite[Lemma 1]{grochenig1996landau}. In general, the kernels satisfy $\{k_x \}_{x \in \R^d}$ satisfy the frame redundancy/density property \eqref{eq:frd} by \cite[Theorem 6.1]{bownik2025redundancy}.

The following theorem is a special case of Theorem \ref{thm:main} phrased in terms of exponential systems. It provides an extension of \cite[Theorem 1.5]{vershynin2000coordinate} from bounded sets to arbitrary measurable sets of finite measure.

\begin{theorem} \label{thm:exponential}
 Let $\Omega \subseteq \R^d$ be a set of finite measure. For every $\ve > 0$, there exists $\Lambda \subseteq \R^d$ satisfying
 $
  D^- (\Lambda) \geq (1-\varepsilon) |\Omega|
 $
and such that $\{ e^{2\pi i \lambda \cdot } \mathds{1}_{\Omega} \}_{\lambda \in \Lambda}$ is a Riesz sequence in $L^2 (\Omega)$. Moreover, the Riesz bounds can be chosen as $c_1 (\varepsilon) |\Omega|$ and $c_2 |\Omega|$ for constants $c_1(\ve), c_2 > 0$, where $c_1 (\ve)$  only depends on $\ve$.
\end{theorem}

\subsubsection*{Gabor systems}
For $z=(x, \xi) \in \R^{2d}$, define the operator $\pi(x, \xi)$ on $L^2 (\R^d)$ by
\[
 \pi(z)f(t)=\pi(x, \xi) f(t) = e^{2\pi i \xi \cdot t} f(t-x), \quad t \in \R^d, \; f \in L^2 (\R^d).
\]
For a unit vector $\phi \in L^2 (\R^d)$, define the map $V_{\phi} : L^2 (\R^d) \to L^2 (\R^{2d}), \; V_{\phi} f(z) = \langle f, \pi(z) \phi \rangle$, $z\in \R^{2d}$. Then, $V_\phi$ is an isometry, and the image space $V_{\phi} (L^2 (\R^d)) \subseteq L^2 (\R^{2d})$ is a reproducing kernel Hilbert space with reproducing kernels given by
\[
 k_{z} := V_{\phi} \pi(z) \phi, \quad z \in \R^{2d}.
\]
Note that $\| k_z \|^2 = 1$ for all $z \in \R^{2d}$. For an arbitrary $g \in L^2 (\R^d)$, the vectors
\[
 g_z := V_{\phi} \pi(z) g, \quad z \in \R^{2d}
\]
form a continuous frame for $V_{\phi} (L^2 (\R^d))$ and satisfy the frame redundancy/density property \eqref{eq:frd} by \cite[Theorem 3]{balan2006density2}. Alternatively, if $\phi \in L^2 (\R^d)$ is the normalized Gaussian $\phi = 2^{d/4} e^{-\pi |\cdot|^2}$ , then the vectors $\{g_z \}_{z \in \R^{2d}}$ satisfy the homogeneous approximation property \eqref{eq:hap} by, e.g., \cite[Lemma 1]{ramanathan1995incompleteness}. 

The following theorem is a direct consequence of Theorem \ref{thm:main} and extends the existence result of Gabor Riesz sequences near the critical density \cite[p. 3789]{casazza2012infinite} from vectors in the modulation space $M^1 (\R^d)$ to arbitrary vectors in $L^2 (\R^d)$.

\begin{theorem}
 Let $g \in L^2 (\R^d)$ be nonzero. For every $\ve > 0$, there exists $\Lambda \subseteq \R^{2d}$ satisfying
 $
  D^- (\Lambda)  \geq 1-\varepsilon
 $
and such that $\{ \pi(\lambda) g \}_{\lambda \in \Lambda}$ is a Riesz sequence in $L^2 (\R^d)$.
\end{theorem}

\subsubsection*{Kernels with nonconstant diagonals} Theorem \ref{thm:main} also applies to various examples for which the associated reproducing kernels do not have a constant diagonal. For example, Theorem \ref{thm:main} applies to spectral subspaces of elliptic differential operators \cite{grochenig2024necessary} and weighted spaces of entire functions \cite{grochenig2019strict}. We refer to these papers for the verifications of the weak localization condition \eqref{eq:wl} and the homogeneous approximation property \eqref{eq:hap}.

\subsection{Lie groups of polynomial growth}
In this section, we consider reproducing kernel Hilbert spaces arising from unitary representations of Lie groups of polynomial growth. The frame and Riesz property of such kernels and associated coherent systems have been studied in, e.g., \cite{fuehr2017density, caspers2023overcompleteness, enstad2025coherent, enstad2025dynamical}.

Let $G$ be a connected noncompact Lie group with left Haar measure $\mu$. The group $G$ is said to be of \emph{polynomial growth} if for every compact unit neighborhood $U \subseteq G$, there exists $C, D > 0$ such that
\[
 \mu(U^n) \leq C n^D
\]
for all $n \in \mathbb{N}$. In this case, a left-invariant Riemannian metric or Carnot-Carath\'eodory metric on $G$ can be shown to satisfy the standing assumptions \eqref{eq:NDB}, \eqref{eq:wad} and \eqref{eq:doubling}, see, e.g., \cite[Corollary 1.6, Section 4.3]{breuillard2014geometry}.

A projective unitary representation $(\pi, \Hpi)$ of $G$ on a separable Hilbert space $\Hpi$ is a measurable map $\pi : G \to \mathcal{U}(\Hpi)$ satisfying $\pi(e) = I_{\Hpi}$ and
\[
 \pi(x) \pi(y) = \sigma(x,y) \pi(xy), \quad x,y \in G,
\]
for a (measurable) function $\sigma : G \times G \to \mathbb{T}$. For a nonzero vector $\eta \in \Hpi$, we define the associated map $V_{\eta} : \Hpi \to L^2 (G)$ by $V_{\eta} f = \langle f, \pi(\cdot) \eta \rangle$. We assume that $\pi$ is irreducible and that there exists $\eta \in \Hpi \setminus \{0\}$ such that $V_{\eta} \eta \in L^2 (G)$, that is, that $\pi$ is irreducible and square-integrable. In this case, there exists $d_{\pi} > 0$, called the formal dimension of $\pi$, such that
\[
 \int_G |\langle f_1, \pi(x) f_2 \rangle |^2 \; d\mu(x) = d_{\pi}^{-1} \| f_1 \|^2 \| f_2 \|^2 \quad \text{for all} \quad f_1, f_2 \in \Hpi,
\]
 see, e.g., \cite{aniello2006square}. In addition, we define the space
\[
 \mathcal{B}_{\pi} := \bigg\{ \eta \in \Hpi : \int_G \sup_{y \in B_1 (e)} |V_{\eta} \eta (xy)|^2 \; d\mu(x) < \infty \bigg\}.
\]
This space can be shown to be dense in $\Hpi$, see, e.g., \cite[Remark 1]{grochenig2008homogeneous}.

If $\eta \in \mathcal{B}_{\pi}$ is nonzero, then the image space $V_{\eta} (\Hpi) \subseteq L^2 (G)$ is a reproducing kernel Hilbert space with reproducing kernels given by
\[
 k_x(y) = \langle \pi(y) \eta , \pi(x) \eta \rangle, \quad x, y \in G.
\]
The kernels $\{k_x \}_{x \in G}$ satisfy the frame measure/density property \eqref{eq:frd} by \cite[Theorem 3.2]{caspers2023overcompleteness} and the homogeneous approximation property \eqref{eq:hap} by \cite[Proposition 2]{grochenig2008homogeneous}.

The following is a direct consequence of Theorem \ref{thm:main}.

\begin{theorem}
 Let $G$ be a connected noncompact Lie group of polynomial growth and let $(\pi, \Hpi)$ be an irreducible, square-integrable projective representation of formal dimension $d_{\pi} > 0$. Let $\eta \in \mathcal{B}_{\pi}$. For every $\varepsilon > 0$, there exists $\Lambda \subseteq G$ satisfying $D^-(\Lambda) \geq (1-\varepsilon) d_{\pi}$ and such that $\{ \pi(\lambda) \eta \}_{\lambda \in \Lambda}$ is a Riesz sequence in $\Hpi$.
\end{theorem}

\section*{Tool and computational resource disclosure}
No large language models have been used for suggesting mathematical arguments or finding references. The mathematical content and the manuscript text are entirely the work of the authors. 

\section*{Acknowledgements}
The first author was partially supported by the NSF grant DMS-2349756.
For J.~v.~V., this research was funded in whole or in part by the Austrian
Science Fund (FWF): 10.55776/PAT2545623.

\bibliographystyle{abbrv}
\bibliography{bib}

\end{document}